\documentclass{amsart}

\usepackage{amsfonts}
\usepackage{fancyhdr}
\usepackage{amsmath}
\usepackage{amssymb}
\usepackage{graphicx}
\usepackage{enumerate}
\usepackage{verbatim}
\usepackage{setspace}
\usepackage{pdflscape}
\usepackage{float}
\usepackage{hyperref}
\usepackage{tikz}
\usetikzlibrary{matrix,arrows}
\usepackage{tikz-cd}
\usepackage{multicol}
\usepackage{color}
\usepackage{bbding}
\usepackage{cite}
\usepackage{ytableau}

\usepackage[capitalise]{cleveref}

\newtheorem{thm}{Theorem}[section]
\newtheorem{lema}[thm]{Lemma}
\newtheorem{prop}[thm]{Proposition}
\newtheorem{cor}[thm]{Corollary}

\newtheorem{question}{Question}[section]
\theoremstyle{definition}
\newtheorem{ex}[thm]{Example}
\newtheorem{rmk}[thm]{Remark}
\newtheorem{defn}[thm]{Definition}

\newcommand{\tp}{\underline{m}}

\newcommand{\lab}{\lambda_{a,b}}
\newcommand{\lcd}{\lambda_D^c}
\newcommand{\dab}{D_{a,b}}
\newcommand{\Mod}[1]{\ \left(\mathrm{mod}\ #1\right)}
\newcommand{\rk}{\textrm{rk}}

\renewcommand{\deg}{\textrm{deg}}
\newcommand{\ddiv}{\textrm{div}}

\newcommand{\PP}{\mathbb{P}}

\newcommand{\RR}{\mathbb{R}}
\newcommand{\ZZ}{\mathbb{Z}}
\newcommand{\TT}{\mathbb{T}}

\newcommand{\cO}{\mathcal{O}}

\newcommand{\cE}{\mathcal{E}}

\newcommand{\Trop}{\operatorname{Trop}}

\begin{document}
\title{Scrollar Invariants of Tropical Curves}
\author{David Jensen}
\address{David Jensen: Department of Mathematics,  University of Kentucky \hfill \newline\texttt{}
\indent 733 Patterson Office Tower,
Lexington, KY 40506--0027, USA}
\email{{\tt dave.jensen@uky.edu}}
\author{Kalila Lehmann}
\address{Kalila Lehmann: Department of Mathematics,  Georgia Insitute of Technology \hfill \newline\texttt{}
\indent 686 Cherry Street,
Atlanta, GA 30332-0160 USA}
\email{{\tt  klehmann6@gatech.edu.}}
\maketitle

\begin{abstract}
We define scrollar invariants of tropical curves with a fixed divisor of rank 1.  We examine the behavior of scrollar invariants under specialization, and provide an algorithm for computing these invariants for a much-studied family of tropical curves.  Our examples highlight many parallels between the classical and tropical theories, but also point to some substantive distinctions.
\end{abstract}

\section{Introduction}

In this note, we initiate a study of scrollar invariants of tropical curves.  Classically, every canonically embedded trigonal curve is contained in a unique rational normal surface.  Any such surface is isomorphic to a Hirzebruch surface
\[
\mathbb{F}_m := \PP(\cO_{\PP^1} \oplus \cO_{\PP^1}(m)),
\]
and the integer $m$ is known as the \emph{Maroni invariant} of the trigonal curve.

More generally, given a curve $X$ and a dominant map $\pi: X \to \PP^1$ of degree $k$, one defines the \emph{Tschirnhausen bundle} of $\pi$ to be the dual of $\cE^{\vee} = \pi_*\cO_X/\cO_{\PP^1}$.  The Tschirnhausen bundle is a vector bundle of rank $k-1$ on $\PP^1$.  As such, it splits into a direct sum of line bundles $\cE = \oplus_{i=1}^{k-1} \cO_{\PP^1} (a_i)$.  The integers $a_i$ are natural invariants of the $k$-gonal curve $X$, known as the \emph{scrollar invariants}.  There are many open questions concerning scrollar invariants.  For example, there is no known classification of which sequences of integers $a_i$ arise as scrollar invariants of $k$-gonal curves.  Even in cases where a curve with given scrollar invariants is known to exist, it is unknown whether the space of such curves is irreducible, or what its dimension is.

In a family of curves of gonality $k$, the scrollar invariants are not lower semicontinuous, but a related set of invariants is.  After ordering the scrollar invariants
\[
a_1 \leq a_2 \leq \cdots \leq a_{k-1},
\]
we define the \emph{composite scrollar invariant} $\sigma_j$ to be the sum of the first $j$ scrollar invariants:
\[
\sigma_j = a_1 + a_2 + \cdots + a_j .
\]
Of course, the scrollar invariants themselves can be recovered from the set of composite scrollar invariants.  The composite scrollar invariants are known to be lower semicontinuous.

In this article, we define tropical analogues of composite scrollar invariants.  Key to our study is the observation that the scrollar invariants are determined by the ranks of the line bundles $\pi^* \cO_{\PP^1}(c)$.  Combining this observation with the Baker-Norine theory of divisors on tropical curves, we obtain definitions of tropical composite scrollar invariants.  We refer the reader to Section~\ref{Sec:Preliminaries} for precise definitions.

We prove that composite scrollar invariants cannot increase under specialization.

\begin{thm}
\label{thm:specialization 1}
Let $X$ be a curve over a nonarchimedean field with skeleton isometric to $\Gamma$, and let $D$ be a divisor of rank 1 on $X$.  Then
\[
\sigma_j (X,D) \geq \sigma_j (\Gamma, \Trop D) \mbox{ for all } j.
\]
\end{thm} 

Having established this relationship between the composite scrollar invariants of a curve and those of its tropicalization, we then compute composite scrollar invariants of certain metric graphs.  Of primary interest to us are the \emph{chains of loops}, a much-studied family of metric graphs that has played a central role in tropical proofs of the Brill-Noether Theorem \cite{CDPR} and the Gieseker-Petri Theorem \cite{tropicalGP}, as well as establishing new results such as the Maximal Rank Conjecture for quadrics \cite{MRC1, MRC2} and an analogue of the Brill-Noether Theorem for curves of fixed gonality \cite{CPJ, DaveDhruv, PfluegerKGonal}.

By varying the edge lengths, we obtain chains of loops of various gonalities.  More precisely, the divisor theory of a chain of loops is determined by its \emph{torsion profile} (see \cref{Def:torsion}).  As the gonality increases, so too does the number of torsion profiles for which the corresponding chain of loops has the given gonality.  We do not have a closed formula for composite scrollar invariants.  Nevertheless, given a torsion profile, we can algorithmically compute the composite scrollar invariants, and we have implemented this algorithm in a Sage program, which can be found on the second author's website:
\begin{center}
\url{https://github.com/kalilajo/numberboxes}.
\end{center}

If $X$ is an algebraic curve and $D$ is a divisor of rank 1 on $X$, then the data of the scrollar invariants is equivalent to that of the sequence of ranks $\rk(cD)$.  More precisely, the sequence of ranks $\rk(cD)$ is a convex, piecewise linear function in $c$, and the scrollar invariants correspond to the ``bends'' between domains of linearity (see \cref{eq:ranksequence}).  For a tropical curve, however, the sequence of ranks is not necessarily convex (see \cref{Ex:Trigonal}).  Because of this, there are several equivalent ways of defining scrollar invariants on algebraic curves that do not agree in the tropical setting.  For example, if $\ell$ is the smallest integer such that $\rk (\ell D) > \ell$, then it is easy to see that $\ell = \sigma_1$.  In \cref{Ex:Trigonal}, however, we exhibit a tropical curve $\Gamma$ and a divisor $D$ of rank 1 such that $\sigma_1 (\Gamma,D)$ is strictly greater than $\ell$.

It is our hope that the study initiated here could be used to resolve outstanding questions concerning scrollar invariants of classical curves.  In order to do this, we would need a lifting result for scrollar invariants.  We pose this as an open question.

\begin{question}
Let $\Gamma$ be a chain of loops, and let $D$ be a divisor of rank 1 on $\Gamma$.  Under what circumstances does there exist a curve $X$, over a nonarchimedean field, with skeleton $\Gamma$ and a rank 1 divisor $D_X$ on $X$ specializing to $D$, such that $\sigma_j (X,D_X) = \sigma_j (\Gamma,D)$?
\end{question}

\subsection*{Acknowledgements}  The first author was supported by NSF DMS-1601896.  The second author was supported by the Graduate Scholars in Mathematics program at the University of Kentucky.  We would like to thank Ralph Morrison and Dhruv Ranganathan for helpful discussions during early stages of this project.  We also thank the anonymous referee for a close reading and many helpful suggestions.

\section{Preliminaries}
\label{Sec:Preliminaries}

\subsection{The Maroni Invariant and Scrollar Invariants}

Let $X$ be a curve of genus $g$ and $\pi: X \to \PP^1$ a dominant map of degree $k \geq 3$.  The map $\pi$ induces a short exact sequence
\[
0 \to \cO_{\PP^1} \to \pi_* \cO_X \to \cE^{\vee} \to 0.
\]
The sheaf $\cE$ is a vector bundle of rank $k-1$ on $\PP^1$, called the \emph{Tschirnhausen bundle} of the map $\pi$.  Since every vector bundle on $\PP^1$ splits as a direct sum of line bundles, we may write
\[
\cE = \bigoplus_{i=1}^{k-1} \cO_{\PP^1} (a_i) .
\]

The integers $a_i$ are known as the \emph{scrollar invariants} of the map $\pi$.  We order them so that
\[
a_1 \leq a_2 \leq \cdots \leq a_{k-1}.
\]
We define the $j$th \emph{composite scrollar invariant} to be the sum of the first $j$ scrollar invariants:
\[
\sigma_j = a_1 + a_2 + \cdots + a_j.
\]

The scrollar invariants determine, and are determined by, the sequence of integers $h^0(X,\pi^*\cO_{\PP^1}(c))$.  Setting $a_0 = 0$, this can be seen by the following calculation:
\begin{align}
\label{eq:ranksequence}
h^0(X,\pi^*\cO_{\PP^1}(c)) &= h^0(\PP^1,\pi_* \cO_X \otimes \cO_{\PP^1}(c)) \\
&= \sum_{i=0}^{k-1} h^0(\PP^1,\cO_{\PP^1}(c-a_i)) \notag \\
&= \sum_{i=0}^{k-1} \max \{ 0, c+1-a_i \}  \notag \\
&= \max \{ (c+1)(j+1) - \sigma_j \} . \notag
\end{align}
Note in particular that $h^0(X,\pi^*\cO_{\PP^1}(c))$ is convex as a function in $c$.

Because $h^0(X,\cO_X) = 1$, we see that each of the scrollar invariants $a_i$ is strictly positive.  Moreover, for $c$ sufficiently large, we have $h^0(X,\pi^*\cO_{\PP^1}(c)) = ck-g+1$, so we see that $\sigma_{k-1} = g+k-1$.

When $k=3$, the scrollar invariants are determined by the single value $a_2 - a_1$, which is known as the \emph{Maroni invariant} of the trigonal curve.  The parity of the Maroni invariant agrees with that of $g$.  The space of trigonal curves with Maroni invariant at least $m$ is known to be nonempty and irreducible if and only if $m \leq \frac{1}{3}(g+2)$ and, except in the case $m=0$, it has codimension $m-1$ in the space of all trigonal curves.

When $k \geq 4$, the situation is more mysterious.  One defines the \emph{Maroni locus} $M(\cE)$ to be the space of $k$-gonal curves with Tschirnhausen bundle isomorphic to $\cE$.  The possible scrollar invariants of tetragonal curves have been classified only recently \cite[Theorem~3.1]{VV24}, and for $k \geq 5$, it is not even known whether $M(\cE)$ is empty for a given a vector bundle $\cE$.

\subsection{Divisor Theory of Metric Graphs}

In this section, we review the theory of divisors on metric graphs.  For more details, we refer the reader to \cite{Baker}.

Recall that a \emph{metric graph} is a compact, connected metric space $\Gamma$ obtained by identifying the edges of a graph $G$ with line segments of fixed positive real length. 

\begin{defn}
A \emph{divisor} $D$ on a metric graph $\Gamma$ is a finite formal $\mathbb{Z}$-linear combination of points of $\Gamma$.  That is, $D=\sum_{v\in \Gamma} D(v)\cdot v$, where $D(v)\in\mathbb{Z}$ is zero for all but finitely many $v$.
\end{defn}

The group of all divisors on a metric graph $\Gamma$ is simply the free abelian group on points of the metric space $\Gamma$, called the \emph{divisor group} $\mathrm{Div}(\Gamma)$ of $\Gamma$.  Divisors on metric graphs should be thought of as the tropical analogues of divisors on algebraic curves.  We now define the tropical analogues of rational functions.

\begin{defn}
A \emph{rational function} on a metric graph $\Gamma$ is a continuous piecewise-linear function $\varphi: \Gamma \to \mathbb{R}$ with integer slopes.  The rational functions on $\Gamma$ form a group under pointwise addition, denoted $\mathrm{PL}(\Gamma)$. Given $\varphi\in \mathrm{PL}(\Gamma)$ and $v\in\Gamma$, we define the \emph{order of vanishing} of $\varphi$
at $v$, $\mathrm{ord}_v(\varphi)$, to be the sum of the incoming slopes of $\varphi$ at $v$.
\end{defn}

Note that $\mathrm{ord}_v(\varphi)$ is nonzero for only finitely many points $v\in\Gamma$.  We define the divisor associated to $\varphi$ as 
\[
\ddiv(\varphi)=\sum_{v\in \Gamma} \mathrm{ord}_v(\varphi)\cdot v .
\]

\begin{defn}
We say that two divisors $D$ and $D'$ on a metric graph $\Gamma$ are \emph{linearly equivalent} if their difference $D-D'$ is equal to $\ddiv(\varphi)$ for some rational function $\varphi\in\mathrm{PL}(\Gamma)$. 
\end{defn}

It is straightforward to show that linear equivalence is in fact an equivalence relation. For our purposes, it suffices to consider linear equivalence classes of divisors. 

A basic invariant of a divisor $D$ is its \emph{degree}, defined to be the integer
\[
\deg(D) = \sum_{v\in \Gamma} D(v).
\]
In analogy with divisors on algebraic curves, we say that a divisor $D$ is \emph{effective} if $D(v) \geq 0$ for all $v \in \Gamma$.  Similarly, we say that a divisor $D$ is \emph{special} if both $D$ and $K_{\Gamma}-D$ are equivalent to effective divisors, where $K_{\Gamma}$ is the canonical divisor
\[
K_{\Gamma} = \sum_{v\in \Gamma}(\mathrm{val}(v)-2)v.
\]

Perhaps the most important invariant of a divisor on a metric graph is its Baker-Norine rank.

\begin{defn}
A divisor $D$ has \emph{rank} at least $r$ if $D-E$ is equivalent to an effective divisor for all effective divisors $E$ of degree $r$. 
\end{defn}

\subsection{Divisors on Chains of Loops}

In Sections~\ref{sec:Tableaux} and~\ref{Sec:Algorithm}, we consider equivalence classes of special divisors on the metric graph pictured in Figure \ref{Fig:CoL}.  This graph, known as the \emph{chain of loops}, has appeared in several articles that use tropical techniques to develop results in algebraic geometry \cite{CPJ, CDPR, tropicalGP, MRC1, MRC2, DaveDhruv, PfluegerKGonal, Pflueger}.  We denote by $v_k$ the point where the $k^{th}$ loop meets a bridge on the left and by $w_k$ the point where the $k^{th}$ loop meets a bridge on the right.  We label edges by their initial and terminal vertices when traversing the loop counter-clockwise.  For example, $w_2v_2$ denotes the top edge of the second loop.
\begin{figure}[h!]
\begin{tikzpicture}

\draw (3,0) circle (1);
\draw (2.3,0) node {\footnotesize $v_1$};
\draw [ball color=black] (2,0) circle (0.5mm);
\draw (3.7,0) node {\footnotesize $w_1$};
\draw [ball color=black] (4,0) circle (0.5mm);

\draw (4,0)--(5,0);
\draw (6,0) circle (1);
\draw [ball color=black] (5,0) circle (0.5mm);
\draw [ball color=black] (7,0) circle (0.5mm);
\draw (5.3,0) node {\footnotesize $v_2$};
\draw (6.7,0) node {\footnotesize $w_2$};
\draw (6,0.75) node {\footnotesize $w_2v_2$};

\draw [dotted] (7,0)--(8,0);
\draw (9,0) circle (1);
\draw [ball color=black] (8,0) circle (0.5mm);
\draw [ball color=black] (10,0) circle (0.5mm);
\draw (8.5,0) node {\footnotesize $v_{g-1}$};
\draw (9.5,0) node {\footnotesize $w_{g-1}$};

\draw (10,0)--(11,0);
\draw (12,0) circle (1);
\draw [ball color=black] (11,0) circle (0.5mm);
\draw [ball color=black] (13,0) circle (0.5mm);
\draw (11.3,0) node {\footnotesize $v_g$};
\draw (12.7,0) node {\footnotesize $w_g$};

\end{tikzpicture}
\caption{A Chain of loops $\Gamma$}
\label{Fig:CoL}
\end{figure}

In this section we summarize the main result of \cite{Pflueger} and draw a few corollaries.

\begin{defn}
\label{Def:torsion}
Let $\ell_i$ denote the length of the $i^{th}$ cycle, and let $\ell(w_iv_i)$ denote the length of the counterclockwise edge from $w_i$ to $v_i$. If $\ell(w_iv_i)$ is an irrational multiple of $\ell_i$, then the \emph{$i^{th}$ torsion order} $m_i$ is 0. Otherwise, $m_i$ is the minimum positive integer such that $m_i\cdot \ell(w_iv_i)$ is an integer multiple of $\ell_i$. We record the torsion order of each loop as the vector $\tp = (m_1, m_2, \ldots, m_g)$, called the \emph{torsion profile} of $\Gamma$. 
\end{defn}

To represent divisors on chains of loops, we use the fact that the Picard group $\textrm{Pic}(\Gamma)$ has a natural coordinate system. Denote by $\langle x\rangle_i$ the point on the $i^{th}$ loop of $\Gamma$ located $x\cdot \ell(w_iv_i)$ units clockwise from $w_i$. Note that $\langle x \rangle_i=\langle y \rangle_i$ if and only if $x\equiv y \Mod {m_i}$. 

By the Tropical Abel-Jacobi theorem \cite{tropicalAJ}, every divisor class $D$ of degree $d$ on $\Gamma$ has a unique \emph{break divisor} representative
\[
D \sim (d-g)w_g + \displaystyle\sum_{i=1}^g \langle\xi_i(D)\rangle_i
\]
for some $\xi_i(D)\in \RR/m_i\ZZ.$ These divisors are our primary object of study.

Recall that a partition $\lambda$ of an integer may be represented as a Young diagram, whose boxes may be filled with symbols to form a Young tableau. We will use the term ``partition" interchangeably with its Young diagram. We denote by $(x,y)$ the box in column $x$ and row $y$ of $\lambda$, where the top left box has coordinates $(0,0).$

\begin{defn}
An \emph{$\tp$-displacement tableau} on a partition $\lambda$ is a function \\$t:\lambda\rightarrow \{1,\ldots, g\}$ such that:
\begin{enumerate}
    \item $t$ increases across each row and column of $\lambda$, and
    \item if $t(x,y)=t(x',y')=i$, then $y-x\equiv y'-x'\Mod{m_i}.$
\end{enumerate}
\label{def:tableau}
\end{defn}
Each such tableau t defines a locus $\TT(t) \subseteq \textrm{Pic}^d(\Gamma)$ homeomorphic to a torus of dimension equal to $g$ minus the number of symbols appearing in $t$.  Specifically,
\[
\TT(t) = \{ D \in \textrm{Pic}^d(\Gamma) \vert \xi_{t(x,y)}(D) \equiv y-x \Mod{m_{t(x,y)}} \text{ for all } (x,y)\in \lambda \}. 
\]
Note that if the function $t$ is not surjective, then there is a symbol $i$ not appearing in the tableau, and a corresponding value $\xi_i$ upon which no restrictions are placed.  

Recall that $W^r_d(\Gamma)$ is the set of all divisor classes of degree $d$ and rank at least $r$ on $\Gamma$.  Pflueger's main result in \cite{Pflueger} is the following.

\begin{thm} \cite{Pflueger}
\label{thm: Pflueger}
Let $\Gamma$ be a chain of loops of genus $g$ and torsion profile $\tp$, and let $r$ and $d$ be positive integers with $r>d-g$.  Let $\lambda$ be the rectangular partition of dimensions $(r+1)\times(g-d+r).$ Then 
\[
W^r_d(\Gamma) = \bigcup_{t}\TT(t),
\]
where $t$ ranges over all $\tp$-displacement tableaux on $\lambda$. 
\end{thm} 

\begin{cor}
\label{cor: k-gonal}
A chain of loops with torsion profile $\tp$ has gonality $k$ if and only if there is an \tp-displacement tableau on a rectangle $\lambda$ of dimensions $(g-k+1) \times 2$ and no such tableau on a rectangle of dimensions $(g-k+2) \times 2$.  
\end{cor}

The following lemma will prove to be a crucial step in our analysis of scrollar invariants on chains of loops.

\begin{lema}
\label{lemma: xi}
 Given a divisor $D$ on $\Gamma$, denote by $\xi_i^c:=\xi_i^c(D)$ the coordinate on the $i^{th}$ loop of $\Gamma$ in the break divisor representative of $cD$. Then $\xi^{c+1}_i = \xi_i^{c}+\xi_i^1 - (i-1)$. It follows by induction on $c$ that $\xi_i^c=c\;\xi_i^1-(c-1)(i-1)$.
\end{lema}

\begin{proof}
By \cite[Remark~3.4]{Pflueger}, the function 
\[
\widetilde{\xi}_i := \xi_i - (i-1)
\]
is linear.  This gives 
\begin{align*}
       \xi_i^{c+1} & =i-1+\widetilde{\xi}_i^{c+1}\\
       & = i-1+\widetilde{\xi}_i^c+\widetilde{\xi}_i^1\\
       & = i-1+ \xi_i^c - (i-1) + \xi_i^1 - (i-1)\\
       & = \xi_i^c + \xi_i^1 - (i-1).
\end{align*}
\end{proof}

\subsection{Specialization}

The theory of divisors on metric graphs informs that of algebraic curves via \emph{specialization}.  Here, we recall the basic properties of specialization.  We refer the reader to \cite{Baker} for details.  Let $K$ be an algebraically closed field that is complete with respect to a nontrivial valuation
\[
\mathrm{val} : X \to \mathbb{R}^* .
\]
Let $X$ be an algebraic curve over $K$.  A \emph{skeleton} of $X$ is a certain type of subset of the set of valuations on the function field $K(X)$ that extend the given valuation on $K$. A skeleton of $X$ is endowed with a topology, giving it the structure of a metric graph.  There is a natural map from $X$ to its skeleton $\Gamma$.  Extending linearly yields the tropicalization map on divisors
\[
\Trop : \mathrm{Div} (X) \to \mathrm{Div} (\Gamma).
\]

The tropicalization map satisfies an important property, known as Baker's Specialization Lemma.

\begin{lema} \cite{Baker}
Let $D_X$ be a divisor on $X$.  Then
\[
\rk (D_X) \leq \rk (\Trop D_X).
\]
\end{lema}

\section{Specialization for Composite Scrollar Invariants}
\label{Sec:Special}

We now define composite scrollar invariants of divisors on metric graphs.

\begin{defn}
Let $\Gamma$ be a metric graph and $D$ a divisor of rank 1 on $\Gamma$.  We define the $j$th \emph{composite scrollar invariant} of the pair $(\Gamma,D)$ to be
\[
\sigma_j (\Gamma,D) := \min \{ m \mid \rk(cD) \geq (c+1)(j+1) -(m+1) \text{ for all } c \} .
\]
\end{defn}

Note that $\rk(cD) \geq c$ for all $c$, with equality if $c=0$, so $\sigma_0 = 0.$  If $\deg (D) = k$, then by Riemann-Roch, we have $\rk(cD) \geq ck-g$ with equality if $c$ is sufficiently large, so $\sigma_{k-1} = g+k-1$.

As mentioned in the introduction, there are several other ways we could define tropical analogues of these invariants.  For example, we could define $\sigma_1$ to be the minimum value of $c$ such that $\rk(cD) > c$.  For algebraic curves, these two definitions of $\sigma_1$ agree because the rank sequence $\rk(cD)$ is convex as a function in $c$.   For metric graphs, however, the rank sequence is not necessarily convex, so these two definitions do not agree.

We now prove a specialization lemma for composite scrollar invariants.

\begin{proof}[Proof of \cref{thm:specialization 1}]
By \cref{eq:ranksequence}, for any value of $j$ we have
\[
\rk(cD) \geq (c+1)(j+1) - (\sigma_j (X,D) + 1).
\]
Simultaneously, by Baker's Specialization Lemma, we have
\[
\rk(cD) \leq \rk(c\Trop D) \mbox{ for all } c.
\]
It follows that
\[
\rk (c\Trop D) \geq (c+1)(j+1) - (\sigma_j (X,D) + 1) \mbox{ for all } c.
\]

Since $\sigma_j (\Gamma, \Trop D)$ is defined to be the minimum value of $m$ such that
\[
\rk (c\Trop D) \geq (c+1)(j+1) - (m+1) \mbox{ for all } c,
\]
we see that
\[
\sigma_j (\Gamma, \Trop D) \leq \sigma_j (X,D) .
\]
\end{proof}

\section{Chains of Loops of Fixed Gonality}
\label{sec:Tableaux}
For the remainder of the paper, we compute composite scrollar invariants for a specific family of tropical curves, the chains of loops.  In this section, we classify chains of loops based on their gonality. 
We begin with the following observation.

\begin{lema}
\label{lemma: Lambda}
The following is the unique displacement tableau $\Lambda$ on the rectangular partition $(g-1)\times 2$.
\begin{center}
    \begin{tabular}{| c | c |}
    \hline
    1&2 \\ \hline
    2&3 \\ \hline
    3&4 \\ \hline
    \vdots&\vdots \\ \hline
    g-2&g-1 \\ \hline
    g-1&g\\
    \hline
    \end{tabular}
\end{center}
\end{lema}

\begin{proof}
The boxes of $\Lambda$ must contain integers between 1 and $g$ so that the entries strictly increase in each row and column.  There cannot be a $g$ in the zeroth column, since the box to the right of it must contain a larger number.  Similarly, there cannot be a 1 in the first column.  This leaves exactly $g-1$ distinct symbols that may appear in each column, which must appear in increasing order.  This yields the above tableau.
\end{proof}

By \cref{lemma: Lambda}, we see that the torsion profile of a hyperelliptic chain of loops is essentially unique.

\begin{cor}
\label{cor: helltab}
A chain of loops $\Gamma$ is hyperelliptic if and only if its torsion profile (termwise) divides $\tp=(0,2,2,\ldots,2,0)$. In this case, there is a divisor $D$ on $\Gamma$ of degree 2 and rank 1 whose corresponding tableau is $\Lambda$.  
\end{cor}
    
\begin{proof}
By \cref{cor: k-gonal}, $\Gamma$ is hyperelliptic if and only if there is an $\tp$-displacement tableau on a rectangle of dimensions $(g-1)\times2$.
   
By \cref{lemma: Lambda}, we see that $\Lambda$ is the unique tableau on a $(g-1)\times2$ rectangle.  Since the symbols 1 and $g$ appear only once, $\Lambda$ imposes no conditions on the torsion of the first or last loops of $\Gamma$.  Each symbol $i$ in the range $1<i<g$ appears twice in $\Lambda$, in boxes $(0,i-1)$ and $(1,i-2)$, which are lattice distance 2 from each other.  Thus we must have $m_i=2$ and the torsion profile of $\Gamma$ is as above.
\end{proof}

We denote by $\lab$ the tableau on the rectangular partition $(g-2) \times 2$ obtained by deleting boxes $(1,a-2)$ and $(0,b-1)$ from $\Lambda$, and sliding the remaining boxes together vertically.  Note that the symbols appearing in these boxes are $a$ and $b$, respectively.  This defines a tableau if and only if $b \geq a-1$.  Tableaux of the form $\lab$ are of interest for the following reason.

\begin{prop}
\label{prop: trig tableaux}
All displacement tableaux on a rectangle $\lambda$ of dimensions $(g-2)\times 2$ are of the form $\lab$ for some $b \geq a-1$.
\end{prop}

\begin{proof} 
Let $t$ be a displacement tableau on $\lambda$.  We must show that $t = \lab$ for some $b \geq a-1$.  Note that $t$ has $g-2$ distinct entries in each column, which must be between $1$ and $g$.  As in \cref{cor: helltab}, there may not be a $g$ in the zeroth column or a $1$ in the first column, so there is exactly one integer ``missing" from each column.  Let $b$ be the integer that is missing from the zeroth column, and let $a$ be the integer that is missing from the first column.  Moreover, note that the missing box in the first column may not be strictly below the missing box in the zeroth column.  From this, we achieve the desired result. 
\end{proof}

\cref{prop: trig tableaux} allows us to classify trigonal chains of loops.

\begin{cor}
\label{cor: trig torsion profile}
A chain of loops $\Gamma$ is trigonal if and only if it is not hyperelliptic, and has torsion profile that (termwise) divides 
\[
\tp=( 0, 2,\ldots, 2, 0, 3, \ldots, 3, 0, 2, \ldots, 2, 0).
\]  
\end{cor}    

\begin{proof}
By \cref{cor: k-gonal}, $\Gamma$ is trigonal if and only if there is an $\tp$-displacement tableau on a rectangle $\lambda$ of dimensions $(g-2)\times 2$ and none on a rectangle of dimensions $(g-1)\times2$. 
By \cref{prop: trig tableaux}, every tableau on $\lambda$ is of the form $\lab$ for some $a$ and $b$.  The tableau $\lab$ imposes no conditions on the torsion of loops 1, $a$, $b$, and $g$, but the torsion of each other loop is determined by the tableau.  In particular, if $i<a$, the symbol $i$ appears twice in $\lab$, both in boxes $(0,i-1)$ and $(1,i-2)$.  These boxes are lattice distance 2 from each other, so we must have $m_i=2$.  In the same way, $m_i=2$ for symbols $i$ in the range $b<i<g$.  Similarly, if $a<i<b$, the symbol $i$ appears in boxes $(0,i-1)$ and $(1,i-3)$, which are lattice distance 3 apart, so $m_i=3$.  
\end{proof}  

For higher gonalities, we have the following natural generalization of of \cref{prop: trig tableaux}. 
    
\begin{prop}
Every displacement tableau on $(g-k+1)\times 2$ may be obtained by removing $k-2$ boxes from each column of $\Lambda$ (as defined in \cref{lemma: Lambda}) in such a way that, above any row, the number of boxes deleted from the left column of $\Lambda$ does not exceed the number of boxes deleted from the right column. 
\label{prop: ktableaux}
\end{prop}
    
\begin{proof}
Consider the result $\lambda$ of removing $k-2$ boxes from each column of $\Lambda$ as described and sliding the remaining boxes together vertically.  This forms a rectangle of dimensions $(g-k+1)\times2$, and the condition on removed boxes guarantees that the entries in each row are increasing.  
     
It remains to show that every displacement tableau $t$ on $\lambda$ can be obtained in this way.  For any such tableau, note that each column of $t$ must have $g-k+1 = g-1-(k-2)$ distinct entries, which must be between 1 and $g$. By the definition of tableaux, there may not be a 1 in the first column of $t$ or a $g$ in the zeroth column, so each column contains all but $k-2$ of the symbols that appear in the corresponding column of $\Lambda$.  In other words, the entries in each column may be obtained by deleting $k-2$ of the entries in the corresponding column of $\Lambda$.  Requiring the entries in each row to increase exactly recovers our condition on the boxes removed, and the result follows. 
\end{proof}

\begin{ex}
\label{Ex:Process}
\cref{fig: Build D} illustrates this process.  The resulting tableau corresponds to a divisor of rank 1 and degree 5 on a chain of loops of genus 15.

\begin{figure}[h]
    \centering
    \begin{tikzpicture}[scale=0.45]
        \draw (2,-3) node {1};
        \draw (2,-4) node {2};
        \draw (3,-3) node {2};
        \draw (2,-5) node {3};
        \draw (3,-4) node {3};
        \draw (2,-6) node {4};
        \draw (3,-5) node {4};
        \draw (2,-7) node {5};
        \draw (3,-6) node {5};
        \draw (2,-8) node {6};
        \draw (3,-7) node {6};
        \draw (3,-8) node {7};
        \draw (2,-9) node {7};
        \draw (2,-10) node {8};
        \draw (3,-9) node {8};
        \draw (2,-11) node {9};
        \draw (3,-10) node {9};
        \draw (2,-12) node {10};
        \draw (3,-11) node {10};
        \draw (2,-13) node {11};
        \draw (3,-12) node {11};
        \draw (2,-14) node {12};
        \draw (3,-13) node {12};
        \draw (3,-14) node {13};
        \draw (2,-15) node {13};
        \draw (2,-16) node {14};
        \draw (3,-15) node {14};
        \draw (3,-16) node {15};

        \foreach \x in {1.5, 2.5, 3.5}
        \draw (\x,-16.5)--(\x,-2.5);
        
        \foreach \y in {-2.5, -3.5, -4.5, -5.5, -6.5, -7.5, -8.5, -9.5, -10.5, -11.5, -12.5, -13.5, -14.5, -15.5, -16.5}
        \draw (1.5,\y)--(3.5,\y);
        
        \draw[->] (4.5,-9.5)--(5.5,-9.5);
        \draw (7,-3) node {1};
        \draw (7,-4) node {2};
        \draw (8,-3) node {2};
        \draw (7,-5) node {3};
        \draw (8,-4) node {3};
        \draw (7,-6) node {4};
        \draw (8,-6) node {5};
        \draw (7,-8) node {6};
        \draw (8,-8) node {7};
        \draw (7,-10) node {8};
        \draw (8,-10) node {9};
        \draw (7,-12) node {10};
        \draw (8,-11) node {10};
        \draw (7,-13) node {11};
        \draw (8,-12) node {11};
        \draw (7,-14) node {12};
        \draw (8,-13) node {12};
        \draw (8,-14) node {13};
        \draw (7,-15) node {13};
        \draw (7,-16) node {14};
        \draw (8,-15) node {14};
        \draw (8,-16) node {15};
        
        \foreach \x in {6.5, 7.5, 8.5}
        \draw (\x,-16.5)--(\x,-2.5);
        
        \foreach \y in {-2.5, -3.5, -4.5, -5.5, -6.5, -7.5, -8.5, -9.5, -10.5, -11.5, -12.5, -13.5, -14.5, -15.5, -16.5}
        \draw (6.5,\y)--(8.5,\y);
        \draw[thick,color=white](6.5,-6.5)--(6.5,-7.5);
        \draw[thick,color=white](6.5,-8.5)--(6.5,-9.5);
        \draw[thick,color=white](6.5,-10.5)--(6.5,-11.5);
        \draw[thick,color=white](8.5,-4.5)--(8.5,-5.5);
        \draw[thick,color=white](8.5,-6.5)--(8.5,-7.5);
        \draw[thick,color=white](8.5,-8.5)--(8.5,-9.5);
        \draw[thick,color=white](7.5,-6.5)--(7.5,-7.5);
        \draw[thick,color=white](7.5,-8.5)--(7.5,-9.5);

        \draw[->] (9.5,-9.5)--(10.5,-9.5);
        \draw (12,-3) node {1};
        \draw (12,-4) node {2};
        \draw (13,-3) node {2};
        \draw (12,-5) node {3};
        \draw (13,-4) node {3};
        \draw (12,-6) node {4};
        \draw (13,-5) node {5};
        \draw (12,-7) node {6};
        \draw (13,-6) node {7};
        \draw (12,-8) node {8};
        \draw (13,-7) node {9};
        \draw (13,-8) node {10};
        \draw (12,-9) node {10};
        \draw (12,-10) node {11};
        \draw (13,-9) node {11};
        \draw (12,-11) node {12};
        \draw (13,-10) node {12};
        \draw (12,-12) node {13};
        \draw (13,-11) node {13};
        \draw (12,-13) node {14};
        \draw (13,-12) node {14};
        \draw (13,-13) node {15};

        \foreach \x in {11.5, 12.5, 13.5}
        \draw (\x,-13.5)--(\x,-2.5);
        
        \foreach \y in {-2.5, -3.5, -4.5, -5.5, -6.5, -7.5, -8.5, -9.5, -10.5, -11.5, -12.5, -13.5}
        \draw (11.5,\y)--(13.5,\y);
    \end{tikzpicture}
    \caption{Making $\lambda_D$ from $\Lambda$}
    \label{fig: Build D}
\end{figure}
\end{ex}

This construction provides a natural classification of the tableaux corresponding to divisors of degree $k$ and rank 1 on chains of loops.  We use similar notation to the trigonal case, denoting by $\lambda_D$ the tableau obtained in this manner corresponding to a divisor $D$ on $\Gamma$.  We associate a Dyck word (which we represent with matched sets of parentheses) to each tableau $\lambda_D$ as follows:  delete boxes from $\Lambda$ to form $D$, from top to bottom. As each box is deleted, add a $\big($ or a $\big)$ to the end of the word if the box is deleted from the first or zeroth column, respectively. 

We say two tableaux are of the same \emph{combinatorial type} if they have the same associated Dyck word. Since it is known that Dyck words are enumerated by the Catalan numbers,  the following is immediate.

\begin{cor}
\label{Cor: Catalan}
The number of combinatorial types of tableaux corresponding to divisors of degree $k$ and rank 1 on chains of loops is equal to the $(k-2)^{nd}$ Catalan number, $C_{k-2}$.
\end{cor}

This result has significant computational implications. In the trigonal case, all tableaux have the same combinatorial type.  In higher gonality cases, the tableau $\lcd$ depends on the \emph{$i$-blocks} of $\tp$, which we now define.

\begin{defn}
\label{def: iblock}
Let $i > 1$ be an integer.  A collection $\{a+1, \ldots, b-1 \}$ of consecutive integers in $\{ 1, \ldots, g \}$ is called an \emph{$i$-block of $\tp$} if 
\begin{enumerate}
\item  $i$ is a multiple of $m_j$ for $a < j < b$, and
\item  $i$ is not a multiple of $m_a$ or $m_b$.
\end{enumerate}
\end{defn}

Each combinatorial type of $\lambda_D$ corresponds to a different distribution of $i$-blocks.  In particular, if the symbol $i$ appears only once in the tableau $\lambda_D$, then the $i$th torsion torsion order $m_i$ is arbitrary.  Otherwise, the $i$th torsion order $m_i$ must divide
\begin{align*}
2 - & \# \left( \text{ symbols } <i \text{ missing from column } 0 \right) \\  
 + & \# \left( \text{ symbols } <i \text{ missing from column } 1 \right) .
\end{align*}

\begin{defn}
Let $\lambda_D$ be a rectangular tableau of dimensions $(g-k+1) \times 2$ containing each of the symbols in $\{ 1, \ldots , g \}$. We say that the torsion profile $\tp$ is \emph{nondegenerate} if it satisfies the following conditions:
\begin{enumerate}
\item  if $i$ appears only once in the tableau $\lambda_D$, then $m_i = 0$, and
\item  otherwise,
\begin{align*}
m_i = 2 - & \# \left( \text{ symbols } <i \text{ missing from column } 0 \right) \\  
 + & \# \left( \text{ symbols } <i \text{ missing from column } 1 \right) .
\end{align*}
\end{enumerate}
\end{defn}

\section{Computing the Rank Sequence for Chains of Loops}
\label{Sec:Algorithm}

Let $D$ be a divisor of rank 1 on a chain of loops $\Gamma$.  The goal of this section is to compute the sequence $\rk(cD)$ as $c$ ranges over all positive integers.  By \cref{thm: Pflueger}, the divisor $cD$ has rank at least $r$ if and only if there exists a tableau $\lcd$ on a rectangle with $r+1$ columns and $g-kc+r$ rows such that $cD\in\TT(\lcd)$. By \cref{lemma: xi}, $cD\in \TT(\lcd)$ if and only if, whenever $\lcd(x,y)=i$, we have
\begin{align}
\label{Eq:Cong2}
y-x \equiv \xi_i^c = c\;\xi_i - (c-1)(i-1) \Mod{m_i}.
\end{align}

To produce the largest possible $\tp$-displacement tableau satisfying this congruence condition, we make use of some original SAGE code available at
\begin{center}
\url{https://github.com/kalilajo/numberboxes}.
\end{center}

In the remainder of this section, we describe the algorithm implemented by this code, prove that the resulting tableaux are optimal, and provide a few corollaries.

\subsection{Algorithm for constructing $\lcd$}

In essence, our algorithm is straightforward.  We attempt to construct a $\tp$-displacement tableau, satisfying the congruence condition, that is as large as possible.  If we fail, we reduce the number of columns by 1 and try again, until we succeed.  In theory, our algorithm does not require us to have constructed the tableau $\lambda_D^{c-1}$.  However, since the rank of $cD$ cannot exceed that of $(c-1)D$ by more than $k$, if we know the rank of $(c-1)D$, then we have bounds on the possibilities for the rank of $cD$, and our algorithm is more efficient.  In practice, therefore, we construct the tableaux recursively, starting with $\lambda_D^2$, then using the dimensions of $\lambda_D^2$ to bound those of $\lambda_D^3$, and so on.  To that end, we make the following definition.

\begin{defn}
\label{Def: LambdacD}
For $c\geq 2$, let 
\[
j:= k - ( \rk(cD) - \rk((c-1)D)).
\]
\end{defn}

Note that $\lcd$ has $j$ fewer rows and $k-j$ more columns than $\lambda_D^{c-1}$.  Identifying $j \in \{1,\ldots,k-1\}$ is the overall goal of our calculation.  Given $\rk((c-1)D)$, we construct $\lcd$ as follows.

\vskip 4 pt

\noindent \textbf{Step 1:  Set $j=1$.}
We begin by setting $j = 1$, and we attempt to construct $\lcd$ so that it has $j$ fewer rows and $k-j$ more columns than $\lambda_D^{c-1}$.

\vskip 4 pt

\noindent \textbf{Step 2:  Start with the diagonal $x+y = 0$.}
To construct $\lcd$, we ``traverse'' each diagonal defined by fixing the sum of the coordinates, beginning with $x+y=0$.

\vskip 4 pt

\noindent \textbf{Step 3:  Traverse the diagonal.}
When traversing a diagonal, we start with its leftmost box.  Each time we arrive at a new box $(x,y)$, we fill it with the smallest $s \in \{ 1, \ldots , g \}$ that is larger than both the entry $\lcd(x,y-1)$ above it and the entry $\lcd(x-1,y)$ to the left of it, and such that \cref{Eq:Cong2} is satisfied.  If there is no value of $s$ such that these conditions hold, we increase the value of $j$ by 1 and return to Step 2.  

If we fill the box $(x,y)$, we proceed to the box $(x+1,y-1)$ above and to the right of the current box, along the same diagonal.  If the box $(x,y)$ is the rightmost box on this diagonal, we increase the sum $x+y$ by 1 and repeat Step 3.  If $(x,y)$ is the bottom right corner of the rectangle, terminate the algorithm and output the rectangular tableau $\lcd$.

\begin{rmk}
Alternatively, one can think of this algorithm in the following way.  Proceeding diagonal by diagonal, one produces the largest non-rectangular tableau possible such that $cD \in \TT (\lcd)$.  One then finds the largest value of $r$ such that this partition contains a rectangle with $r+1$ columns and $g-3c+r$ rows.
\end{rmk}

\subsection{Verifying the algorithm}

We apply this algorithm recursively to find the largest tableau $\lcd$ such that $cD \in \TT (\lcd)$ for each value of $c$.  It remains to show the tableaux generated by this algorithm are optimal.

\begin{prop}
Let $t$ be a tableau such that $cD\in\TT(t)$. Then $t(x,y)\geq\lcd(x,y)$ for all $x,y$. 
\end{prop}

\begin{proof}
We proceed by induction. The base case, $t(0,0)\geq 1$, is immediate. We assume that $t(x',y')\geq \lcd(x',y')$ for all $x'$, $y'$ such that either $x'<x$, $y'\leq y$ or $x'\leq x$, $y'< y$ and show that $t(x,y) \geq \lcd(x,y)$.  By construction, $\lcd (x,y)$ is the smallest symbol greater than both $\lcd(x-1,y)$ and $\lcd(x,y-1)$ that satisfies \cref{Eq:Cong2}.  Our inductive hypothesis implies that $t(x,y)$ must satisfy these conditions as well.  We must therefore have $t(x,y)\geq \lcd(x,y)$. 
\end{proof}

\subsection{A Chain of Loops with Generic Scrollar Invariants}

We make a simple observation on the output of our algorithm.

\begin{lema}
\label{lemma: i-count}
Suppose that the torsion profile $\tp$ is nondegenerate.  If $\lcd(x,y)$ and $\lambda_D^c(x,y) + i-1$ are in the same $i$-block, then
\[
\lcd(x+1,y) \geq \lcd(x,y) + i-1.
\]
\end{lema}    
    
\begin{proof}
By definition, we have
\[
y - x \equiv \xi_{\lcd(x,y)} \equiv c\xi_{\lambda_D(x,y)}^1 - (c-1)(i-1) \Mod{i} .
\]
Since $\tp$ is nondegenerate and $\lambda_D(x,y)$ and $\lambda_D(x,y) + i-1$ are in the same $i$-block, we see that
\[
\xi_{\lambda_D(x,y)+j}^1 = \xi_{\lambda_D(x,y)}^1 + j \mbox{ for all } 0 \leq j \leq i-1,
\]
so
\[
\xi_{\lambda_D(x,y)+j}^c \equiv \xi_{\lambda_D(x,y)}^c + j \Mod{i}
\]
for all $j$ in the same range.  It follows that $i-1$ is the smallest value of $j$ such that
\[
\xi_{\lcd(x,y)+j} \equiv \xi_{\lcd(x,y)} -1 \Mod{i}.
\]
We therefore see that $\lcd (x+1,y) \geq \lcd (x,y) + i-1$. 
\end{proof}  

As a consequence, we see that there is a torsion profile that maximizes the composite scrollar invariants.  The torsion profile below corresponds to the tableau where the symbols $g-k+2, \ldots , g$ are missing from column zero, and the symbols $1, \ldots , k-1$ are missing from column one.  We note that this torsion profile has been used in several papers to examine the behavior of general curves of gonality $k$ \cite{CPJ, DaveDhruv, PfluegerKGonal}.  Corollary~\ref{Cor:Generic} provides further evidence that this chain of loops behaves like a general curve of gonality $k$, as it has the scrollar invariants of a general curve.

\begin{cor}
\label{Cor:Generic}
Suppose
\[
\tp = (0,\ldots,0,k,\ldots,k,0,\ldots,0).
\]
Then $\rk(cD) = c$ for all $c$ such that $g > c(k-1)$.  
\end{cor}

\begin{proof}
Suppose that $\lcd$ has more than $c+1$ columns.  By \cref{lemma: i-count}, $\lcd(c+1,0) \geq (c+1)(k-1)$.  It follows that 
\[
\lcd(c+1,g-c(k-1)) \geq g - c(k-1) + (c+1)(k-1) = g + k - 1 > g,
\]
which is impossible.  Thus $\lcd$ has at most $c+1$ columns, and $\rk(cD) = c$. 
\end{proof}

\subsection{Trigonal Chains of Loops}

If the torsion profile is more exotic, then the composite scrollar invariants can vary in interesting ways.  We illustrate this phenomenon using an example when $k=3$.  We first describe aspects of the algorithm that are common to all trigonal chains of loops, and then specialize to a specific case.

There are several aspects of the trigonal case that greatly simplify our procedure.  First, by \cref{cor: trig torsion profile}, there is only one combinatorial type of trigonal chain of loops.  Second, $\sigma_1$ is the only nontrivial composite scrollar invariant.  Moreover, to compute $\sigma_1 (\Gamma,D)$, it suffices to compute the smallest integer $n$ such that $K_{\Gamma} - nD$ is not effective.  This is because, by Riemann-Roch, we have
\begin{align*}
\rk (nD) &= 3n-g, \\
\rk (cD) & = \rk ((c-1)D) + 3 \text{ for all } c > n,  \\
\rk (cD) & \leq \rk((c-1)D) + 2 \text{ for all } c \leq n. 
\end{align*}
It follows from this that $\rk(cD) - 2c$ is minimized when $c = n$, hence:
\[
\sigma_1 (\Gamma, D) = \min \{ m \mid \rk(cD) \geq 2c+1-m \text{ for all } c \} = g+1-n.
\]
For this reason, in the trigonal case it is not necessary to compute the full sequence of ranks $\rk(cD)$ -- it suffices to compute only the invariant $n$.  In our example, however, we do compute the full sequence of ranks, for several reasons.  First, for higher gonality, it is necessary to compute the full sequence, and we wish to illustrate the method.  Second, we highlight some of the pathologies that arise in the tropical setting.  For example, the sequence of ranks is not convex, as it is in classical algebraic geometry.

Fix integers $a$ and $b$, and let $\lab$ be the displacement tableau of dimensions $(g-2) \times 2$ defined in Section~\ref{sec:Tableaux}.  We further assume that $a < b$ and that $m_a = m_b = 0$.  Note that these assumptions exclude some chains of loops of gonality 3.  Among the cases we do not consider, the case where $a=b$ is particularly interesting, as there are infinitely many divisors of degree 3 and rank 1, and they do not all have the same scrollar invariants.  Let $\Gamma$ be a chain of loops with torsion profile $\tp$  and let $\dab \in \mathbb{T}(\lab)$ be a divisor class on $\Gamma$ of degree 3 and rank 1.  Note that since $a \neq b$, there is a unique divisor class $\dab \in \mathbb{T}(\lab)$.

As above, the divisor $c\dab$ has rank at least $r$ if and only if there exists a tableau $\lab^c$ on a rectangle with $r+1$ columns and $g-3c+r$ rows such that $c\dab \in \mathbb{T}(\lab^c)$.  In particular, $n$ will be the smallest positive integer such that no such tableau $\lab^n$ exists.  By \cref{lemma: xi}, $c\dab \in \mathbb{T}(\lab^c)$ if and only if, whenever $\lab^c (x,y) = i,$ we have
\begin{align}
\label{Eq:Cong}
y-x = \left\{ \begin{array}{ll}
i-1 \Mod{m_i} & \textrm{if $i < b$,} \\
i-1-3c \Mod{m_i} & \textrm{if $i \geq b$.}
\end{array} \right.
\end{align}
Note in particular that if $i \leq b$, then the congruence conditions above are independent of $c$.  We now proceed to a concrete example.

\begin{ex}
\label{Ex:Trigonal}
Let $g=12$, $a=2$, and $b=4$.  We will use the algorithm to compute the sequence of ranks $\rk(c\dab)$ for all $c$.  First, we build $\lab^2$.  We naively assume $\lab^2$ has $j=1$ fewer rows and $3-j=2$ more columns than $\lab$.  We traverse and fill the diagonals as in steps 2 and 3 of the algorithm.  While doing this, we may only place a symbol $s$ in box $(x,y)$ if \cref{Eq:Cong} is satisfied.  By the remark following \cref{Eq:Cong}, if $i \leq b$ and the symbol $i$ appears in box $(x,y)$ in $\lab$, then it will appear in this same box in $\lab^c$ for all $c$.  If $b < i < g$, however, then $m_i = 2$ and the boxes $(x,y)$ in $\lab^c$ containing the symbol $i$ depend on the parity of $c$.  Using this, we traverse and fill the diagonals of a $9 \times 4$ tableau as we are able. The result is shown in \cref{fig: lambda_D^2 j=1}.

\begin{figure}[h]
    \centering
   \begin{ytableau}
1 & 3 & 5 & 6 \\
2 & 5 & 6 & 7 \\
3 & 6 & 7 & 8 \\
6 & 7 & 8 & 9 \\
7 & 8 & 9 & 10 \\
8 & 9 & 10 & 11 \\
9 & 10 & 11 &  \\
10 & 11 & 12 & \\
11 &  & &
\end{ytableau}
    \caption{Attempting to build $\lab^2$ with $j=1$}
    \label{fig: lambda_D^2 j=1}
\end{figure}

We see that this attempt was unsuccessful, as there were not enough symbols to fill the whole tableau.  We therefore repeat this process with a tableau of dimensions $8 \times 3$ -- that is, assuming $j=2$.  This time, we are successful.  The resulting tableau is shown in \cref{fig: lambda_D^2 j=23}.  Since the algorithm terminates when $j=2$, we see that $\rk(2\dab) = \rk(\dab) + (3-2) = 2$.  Note that this tableau is the restriction of the previous one to a rectangle with one fewer row and one fewer column.  Our procedure restricts to smaller and smaller rectangles until every box is filled.  In this way, we see that this process is equivalent to building the largest non-rectangular tableau possible, then finding the largest rectangular tableau that it contains. 

\begin{figure}[h]
    \centering
       \begin{ytableau}
1 & 3 & 5 \\
2 & 5 & 6 \\
3 & 6 & 7 \\
6 & 7 & 8 \\
7 & 8 & 9  \\
8 & 9 & 10 \\
9 & 10 & 11 \\
10 & 11 & 12 
\end{ytableau}
    \caption{The tableau $\lab^2$.}
    \label{fig: lambda_D^2 j=23}
\end{figure}

Now that we have built $\lab^2$, we build $\lab^3$ using the same procedure.  We then build $\lab^4$ , $\lab^5$, and finally, $\lab^6$.  The resulting tableaux are pictured in Figure~\ref{Fig:HigherMultiples}.  From these tableaux, we see that the sequence of ranks $\rk(c\dab)$ is given by $1, 2, 3, 5, 6, 8, 9, 12 \ldots$  Note that
\[
\rk(5\dab) - \rk(4\dab) = 1 < 2 = \rk(4\dab) - \rk(3\dab),
\]
hence the sequence of ranks is not convex.

\begin{figure}[h]
    \centering
       \begin{ytableau}
1 & 3 & 6 & 7 \\
2 & 6 & 7 & 8 \\
3 & 7 & 8 & 9 \\
5 & 8 & 9 & 10 \\
8 & 9 & 10 & 11  \\
9 & 10 & 11 & 12 
\end{ytableau} \hskip2ex
\begin{ytableau}
1 & 3 & 5 & 6 & 7 & 8 \\
2 & 5 & 6 & 7 & 8 & 9 \\
3 & 6 & 7 & 8 & 9 & 10 \\
6 & 7 & 8 & 9 & 10 & 11 \\
7 & 8 & 9 & 10 & 11 & 12
\end{ytableau} \vskip2ex

\begin{ytableau}
1 & 3 & 6 & 7 & 8 & 9 & 10 \\
2 & 6 & 7 & 8 & 9 & 10 & 11 \\
3 & 7 & 8 & 9 & 10 & 11 & 12
\end{ytableau} \hskip2ex
\begin{ytableau}
1 & 3 & 5 & 6 & 7 & 8 & 9 & 10 & 11 \\
2 & 5 & 6 & 7 & 8 & 9 & 10 & 11 & 12
\end{ytableau}
    \caption{The tableaux $\lab^3$, $\lab^4$, $\lab^5$, and $\lab^6$.}
    \label{Fig:HigherMultiples}
\end{figure}

From this sequence of ranks, we see that $n=7$, hence $\sigma_1 (\Gamma, \dab) = 6$.  If we let $\ell$ be the smallest integer such that $\rk (\ell \dab) > \ell$, then $\ell = 4$.  As discussed in the introduction, on an algebraic curve, the invariants $\ell$ and $\sigma_1$ are equal, but here, on this tropical curve, they disagree.  In general, if $X$ is a curve with skeleton $\Gamma$ and $D_X$ is a divisor of rank 1 on $X$ specializing to $\dab$, then the invariant $\ell$ is a lower bound on $\sigma_1 (X,D_X)$.  More precisely, we have
\[
\ell \leq \sigma_1 (\Gamma,\dab) \leq \sigma_1 (X,D_X),
\] 
and this example shows that the first inequality can be strict.  In other words, $\sigma_1 (\Gamma,D)$ is a better bound than the invarant $\ell$.  From this, we see that the Maroni invariant of $X$ is at most 2.
\end{ex}

\bibliographystyle{maastyle}
\bibliography{References}

\end{document}